\documentclass{amsart}
\usepackage{amsmath,amssymb,amsthm}
\usepackage{longtable}
\usepackage{a4wide}
\usepackage{epsf,stmaryrd,amscd} 
\usepackage{dsfont,mathtools}
\usepackage{xspace}
\usepackage{graphicx}
\usepackage{latexsym}

\newcommand{\fa}{\mathfrak a} 
\newcommand{\fb}{\mathfrak b} 
\newcommand{\fc}{\mathfrak c} 
\newcommand{\fg}{\mathfrak g} 
\newcommand{\fh}{\mathfrak h} 
\newcommand{\fn}{\mathfrak n} 
\newcommand{\fl}{\mathfrak l} 
 
\newcommand{\fz}{\mathfrak z} 
\newcommand{\fsl}{\mathfrak{sl}} 
\newcommand{\fso}{\mathfrak{so}} 
\newcommand{\cod}{{\mathrm{codim}}} 
\newcommand{\Orb}{{\mathrm{Orb}}}

\newcommand{\bK}{{\mathbb K}}

\newcommand{\bR}{{\mathbb R}} 
\newcommand{\bZ}{{\mathbb Z}} 
\newcommand{\bF}{{\mathbb F}} 
 
\newcommand{\cA}{{\mathcal A}} 
\newcommand{\cB}{{\mathcal B}} 
\newcommand{\cF}{{\mathcal F}} 
\newcommand{\cG}{{\mathcal G}} 
\newcommand{\cL}{{\mathcal L}} 
\newcommand{\cS}{{\mathcal S}} 
\newcommand{\cT}{{\mathcal T}} 

\newcommand{\va}{\ensuremath{\mathbf{a}}\xspace}
\newcommand{\ve}{\ensuremath{\mathbf{e}}\xspace}

\newcommand{\vh}{\ensuremath{\mathbf{h}}\xspace}
\newcommand{\vk}{\ensuremath{\mathbf{k}}\xspace}

\newcommand{\vx}{\ensuremath{\mathbf{x}}\xspace}
\newcommand{\vy}{\ensuremath{\mathbf{y}}\xspace}
\newcommand{\vz}{\ensuremath{\mathbf{z}}\xspace}

\newcommand{\tm}{\ensuremath{\mathrm{m}}\xspace}
\newcommand{\tn}{\ensuremath{\mathrm{n}}\xspace}
\newcommand{\ts}{\ensuremath{\mathrm{s}}\xspace}

\newtheorem{defn}{Definition}[section]
\newtheorem{theorem}[defn]{Theorem}

\newtheorem{lemma}[defn]{Lemma}

\newtheorem{prop}[defn]{Proposition}

\title[Subgroup Growth]{Subgroup Growth in Some Profinite Chevalley Groups}

\author{Inna Capdeboscq}
\email{I.Capdeboscq@warwick.ac.uk, \ K.Kirkina@warwick.ac.uk, \ D.Rumynin@warwick.ac.uk}
\author{Karina Kirkina}
\author{Dmitriy Rumynin}
\date{November 2, 2016}

\begin{document}
\begin{abstract}
In this article we improve the known uniform bound for subgroup growth 
of Chevalley groups $\mathbf{G}(\bF_p[[t]])$.
We introduce 
a new parameter, the ridgeline number $v(\mathbf{G})$, 
and give 
new bounds for the subgroup growth of $\mathbf{G}(\bF_p[[t]])$
expressed through $v(\mathbf{G})$. We achieve this by deriving a new
estimate
for the codimension of $[U,V]$ where $U$ and $V$ are vector
subspaces in the Lie algebra of $\mathbf{G}$.
\end{abstract}

\maketitle

%\section{Introduction}

For a finitely generated group $G$, let $a_n(G)$ be the number of subgroups of $G$ of index $n$ and $s_n(G)$ the number of subgroups of $G$ of index at most $n$.  The ``subgroup growth" of $G$ is the asymptotic behaviour of the sequence $(a_n(G))_{n\in\mathbb{N}}$. It  turns out that the subgroup growth and structure of $G$ are not unrelated, and in fact the latter can sometimes be used as a characterisation of $G$.
(For complete details we refer the reader to a book by Lubotzky and
Segal 
\cite{kn::LS}.)

For example, Lubotzky and Mann \cite{kn::LM91} show that a group
$G$ is $p$-adic analytic if and only if there exists a constant $c>0$
such that $a_n(G)<n^c$. 
This inspiring result is followed by Shalev who 
proves that if $G$ is a pro-$p$ group  for which $a_n(G)\leq n^{c\log_pn}$
for some constant $c<\frac{1}{8}$, then $G$ is $p$-adic analytic.

Answering a question of Mann, Barnea and Guralnick investigate the
subgroup growth of $SL_2^1(\mathbb{F}_p[[t]])$ for $p>2$, and  show
that the supremum of the set of all those $c$
that a pro-$p$ group $G$ is $p$-adic analytic provided 
that $a_n(G)<n^{c\log_pn}$, is no bigger than $\frac{1}{2}$.
%$\sup \{ c\}$ for which a pro-$p$ group $G$ is $p$-adic analytic provided that $a_n(G)<n^{c\log_pn}$, is no bigger than $\frac{1}{2}$.
Thus one may see that not only the growth type, but also the precise values of the constants involved are important when studying the connection between subgroup growth and the  structure  of a group.

Later on Lubotzky and Shalev pioneer a study of the so-called
$\Lambda$-standard groups \cite{kn::LSh}.
A particular subclass of these groups are $\Lambda$-perfect groups for
which 
they show existence of a constant $c>0$ such that
$$a_n(G)<n^{c\log_pn}.$$ 
An important subclass of those groups are the congruence subgroups of Chevalley groups over $\bF_p[[t]]$. 
Let $\mathbf{G}$ be a simple simply connected Chevalley group scheme,
$G(1)$ the first congruence subgroup of $\mathbf{G}(\mathbb{F}_p[[t]])$. 
Ab\'ert, Nikolov and Szegedy show that if 
%$m$ is the dimension of the simple Lie algebra $\fg$ of type $\mathbf{G}$, then
$m$ is the dimension of $\mathbf{G}$, then
$$s_{p^k}(G(1))\leq p^{\frac{7}{2}k^2+mk},$$
that is,  
$s_n(G(1))\leq n^{\frac{7}{2}\log_pn+m}$
\cite{kn::ANS}.

%In this short article we improve the result of \cite{kn::ANS}. To state the result we need to discuss some interesting parameters
%of the Lie algebra $\fg$ of the group $G$.
%Let $l$ be the rank of $\fg$, $n$ the dimension, and $k$ the maximal dimension of the centraliser
%of a noncentral element $\vx\in\fg$. We also fix an invariant bilinear form $\eta=\langle\, ,\,\rangle$ on $\fg$.
%Let $\fg^0$ be its kernel and $r$ the dimension of $\fg^0$.

In this article we improve their estimates (cf. \cite{kn::ANS}). 
The Lie algebra $\fg_\bZ$ of $\mathbf{G}$ is defined over integers.
Let $\bK$ be a field  of characteristic $p$, which could be either zero or a prime.
To state the results we now  introduce a new parameter 
of the Lie algebra $\fg\coloneqq\fg_\bZ\otimes_\bZ \bK$. 
%To state the results we now  introduce a new parameter
%of the Lie algebra $\fg$ of $\mathbf{G}$. Let $\bK$ be the field
%of definition of $\fg$.
%Let $p$ be the characteristic of $\bK$, which could be either zero or a prime.
We fix an invariant bilinear form $\eta=\langle\, ,\,\rangle$ on $\fg$
of maximal possible rank. 
Let $\fg^0$ be its kernel.
Notice that the nullity $r\coloneqq\dim\fg^0$
of $\eta$ is independent of the choice of $\eta$.

\begin{defn}
\label{defn::v}
%Let $\fg$ be a simple Lie algebra of type $\mathbf{G}$ over a field
%$\bK$ of characteristic $p$ ($p$ is  either zero or a prime).
Let $l$ be the rank of $\fg$, $m$ its dimension and $s$ the maximal dimension of the centraliser
of a non-central element $\vx\in\fg$. 
%Fix an invariant bilinear form $\eta=\langle\, ,\,\rangle$ on $\fg$.
%Let $\fg^0$ be its kernel with $r\coloneqq \dim\fg^0$.
%
We define  the {\em ridgeline number} of $\fg$ as
$$v (\mathbf{G}) = v(\fg) \coloneqq  \frac{l}{m-s-r}.$$
\end{defn}

We discuss ridgeline numbers in Section~\ref{rid_small}.
The values of $v(\fg)$ can be found in the table in Appendix~\ref{table1}.

\begin{defn}
\label{defn::notsobad} The positive characteristic $p$ 
of the field $\bK$ is
called {\em good} if 
$p$ does not divide the coefficients
of the highest root.
The positive characteristic $p$ 
of the field $\bK$ is
called {\em very good} if $p$ is good and $\fg$ is simple.
We call the positive characteristic $p$ {\em tolerable} 
if any proper ideal of $\fg$ is contained in its centre.
\end{defn}

We discuss these restrictions in Section~\ref{rid_small}.
We may now state our main result.

\begin{theorem}
\label{theorem::1}
%et $\mathbf{G}(\mathbb{F}_p[[t]])$ be a simple simply connected
%hevalley group of rank $l\geq 2$ defined over $\mathbb{F}_p[[t]]$
%here $p$ is a tolerable prime. 
Let $\mathbf{G}$ be a simple simply connected
Chevalley group scheme of rank $l\geq 2$.
Suppose $p$ is a tolerable prime for $\mathbf{G}$. 
Let  $G(1)$
be the  first congruence subgroup of  $\mathbf{G}(\mathbb{F}_p[[t]])$, 
that is $G(1)=\ker(\mathbf{G}(\mathbb{F}_p[[t]])\twoheadrightarrow 
\mathbf{G}(\mathbb{F}_p))$.
%, and $\fg$  the Lie algebra of $\mathbf{G}$. 
If $m\coloneqq \dim\mathbf{G}$,  then 
 $$a_{p^k}(G(1))\leq  p^{\frac{(3+4v(\fg))}{2}k^2+(m-\frac{3}{2}-2v(\fg))k}.$$

If $l=2$ and $p$ is very good, then  a stronger estimate holds:
$$a_{p^k}(G(1))\leq  p^{\frac{3}{2}k^2+(m-\frac{3}{2})k}.$$ 
\end{theorem}
\vskip 3mm
Notice 
%that as one can see from the table above, the biggest possible value of $v(\fg)$ is $\frac{1}{2}$, making $\frac{3+4v(\fg)}{2}\leq \frac{5}{2}$.
that as one can see from the table in Appendix~\ref{table1}, 
with one exception 
(when $\mathbf{G}=A_l$  and $p$ divides $l+1$) the biggest possible value of $v(\fg)$ 
is $\frac{1}{2}$ ($v(\fg)\leq\frac{2}{3}$ in that special case). 
This makes  $\frac{3+4v(\fg)}{2}\leq \frac{5}{2}$ ($\frac{3+4v(\fg)}{2}\leq \frac{17}{6}$ correspondingly).

 Our proof of Theorem~\ref{theorem::1} in many ways follows the ones of  Barnea and Guralnick and of Ab\'ert, Nikolov and Szegedy. 
The improvement in the result is due to the following new estimates.

\begin{theorem}
\label{theorem::2}
%Let $\bK$ be a field of characteristic $p$,
%$\mathbf{G}$ a simple simply connected algebraic group
% of rank $l$ defined over $\bK$,
%$\fg$  its Lie algebra.
Let $\fa$ be a Lie algebra over a field $\bK$.
%Let $\fa$ be an $m$-dimensional Lie algebra over a field $\bK$.
Suppose that 
the Lie algebra
$\fg=\fa\otimes_{\bK} \overline{\bK}$
is a Chevalley Lie algebra of rank $l\geq 2$
and that the characteristic of $\bK$ is zero or tolerable.
Then for any two subspaces $U$ and $V$ of $\fa$, we have
$$
\cod ([U,V]) \leq (1+v(\fg))(\cod (U) + \cod (V)).
$$

%Moreover, if $l=2$ and $\fg$ is of type $A_2, B_2$ or $G_2$,
If $l=2$ and 
the characteristic of $\bK$ is zero or very good,
%$p$ is very good for $\fg$, 
a stronger result holds:
$$
\cod ([U,V]) \leq \cod (U) + \cod (V).
$$
\end{theorem}

We conjecture that the second estimate holds for any Lie algebra
$\fg$ (if  the characteristic of $\bK$ is zero or very good).

\section{Proof of Theorem~\ref{theorem::1}}
The proof of Theorem~\ref{theorem::1} 
relies on Theorem~\ref{theorem::2} 
that will be proved later.
We follow  
Ab\'ert, Nikolov and Szegedy  
\cite[Theorem 2]{kn::ANS}
and
Barnea, Guralnick
\cite[Theorem 1.4]{kn::BG}.

Suppose that hypotheses of Theorem~\ref{theorem::1} hold.
We start with the following observation
(cf. \cite[Corollary 1 and Lemma 1]{kn::ANS} and 
\cite[Lemma 4.1]{kn::LSh}
).
\begin{lemma}
\label{lemma::generators}
If $H$ is an open subgroup of $G(1)$ and $d(H)$ 
is the minimal number of generators of $H$, then
$$d(H)\leq m+(3+4v(\fg))\log_p|G(1):H|.$$
Moreover, if $l=2$
and $p$ is very good, 
then $$d(H)\leq m+3\log_p|G(1):H|.$$
\end{lemma}
Notice that in the second case
$\fg=A_2$, $C_2$ or $G_2$
and
$m=8$, $10$ or $14$ correspondingly.
%The proof of this statement is based in the proof of  Lemma 4.1 of 
%\cite{kn::LSh} and on the modification of Lemma 1 on p. 373 of 
%\cite{kn::ANS}.
%For completeness we provide here its condensed version  with the 
%adjustment leading to the improvement of the estimate.
\begin{proof}
First of all recall that $d(H)=\log_p|H:\Phi(H)|\leq \log_p|H:H'|$
where  
$\Phi(H)$ is the Frattini subgroup. 
Because of the correspondence between the open subgroups of $G(1)$ and  
subalgebras of its graded Lie algebra  $\mathcal{L}=L(G(1))$ (see 
\cite{kn::LSh}), $\log_p|H:H'|\leq\dim\mathcal{H}/\mathcal{H}'$ where 
$\mathcal{H}=L(H)$ is the corresponding subalgebra of $\mathcal{L}$.
Hence it suffices to show that $$\dim\mathcal{H}/\mathcal{H}'\leq 
m+(3+4v(\fg))\dim\mathcal{L}/\mathcal{H}$$
in the general case, and  that $$\dim\mathcal{H}/\mathcal{H}'\leq 
m+3\dim\mathcal{L}/\mathcal{H}$$ in the very good rank $2$ case.

Recall that the graded Lie algebra $\cL$
is isomorphic to
$\fg\otimes_{\bF} t\bF[t]$ where $\bF=\bF_p$.
%We now follow word to word the argument of the proof of Lemma 1 of 
%\cite{kn::ANS} until $(\*)$.
%
Since every element $a\in\mathcal{L}$ can be uniquely written as 
$a=\Sigma_{i=1}^{\infty} a_i\otimes t^i$ with $a_i\in \fg$,
one can define $l(a)\coloneqq a_s$ where $s$ is the smallest integer 
such that $a_s\neq 0$, 
and in this case $s\coloneqq \deg(a)$.
Now set $$H_i\coloneqq \langle l(a) \mid a\in\mathcal{H} \ \mbox{with} \ \deg(a)=i\rangle.$$
Observe that 
$H_i=\{ l(a) \mid a\in\mathcal{H} \ \mbox{with}
\ \deg(a)=i\}\cup\{ 0\}.$
Then $\dim \mathcal{L}/\mathcal{H}=\Sigma_{i=1}^{\infty} \dim \fg/H_i$, and 
this sum is finite as the left hand side is finite.
Then $$[H_i\otimes t^i, H_j\otimes t^j]\subseteq [H_i,H_j]\otimes 
t^{i+j}\subseteq H'_{i+j}\otimes t^{i+j}$$ where $H'_{i+j}\coloneqq \langle l(a) 
\mid a\in\mathcal{H}' \ \mbox{with}\ \deg(a)=i+j\rangle$,
and so $\dim \fg/[H_i, H_j]\geq \dim \fg/H'_{i+j}$. Adding up these 
inequalities for $i=j$ and $i=j+1$ we get
$$\dim \mathcal{L}/\mathcal{H}'=\Sigma_{i=1}^{\infty} \dim 
L/H_i'\leq\dim \fg+\Sigma_{1\leq i\leq j\leq i+1}^{\infty} \dim \fg/[H_i, 
H_j]. \ \ $$%(\*)$$
Now we use the estimates of Theorem~\ref{theorem::2}:
$$\dim \mathcal{L}/\mathcal{H}'\leq m+\Sigma_{1\leq i\leq j\leq 
i+1}^{\infty} \alpha (\dim \fg/H_i+\dim \fg/H_j)\leq 
m+4\alpha \dim \mathcal{L}/\mathcal{H},$$
where $\alpha =1+v(\fg)$ or $1$ depending on the rank of 
$\fg$ and $p$. The result follows immediately.
\end{proof}

%We may now apply Lemma 4.1 of \cite{kn::LSh}, which says that 
Now we apply an estimate \cite[Lemma 4.1]{kn::LSh}: 
$a_{p^k}(G(1))\leq p^{g_1+...+g_{p^{k-1}}}$ 
where
$$
g_{p^i}=g_{p^i}(G(1))=\max\{d(H)\mid H\leq_{open}G(1), |G(1):H|=p^i\}.
$$
Using Lemma~\ref{lemma::generators}, in the general case ($l\geq 
2$) we have
$$a_{p^k}(G(1))\leq p^{\Sigma_{i=0}^{i=k-1} m+(3+4v(\fg))i}= 
p^{\frac{(3+4v(\fg))}{2}k^2+(m-\frac{3}{2}-2v(\fg))k}.$$
For  $l=2$ and very good $p$,  Lemma~\ref{lemma::generators}  gives us
$$a_{p^k}(G(1))\leq p^{\Sigma_{i=0}^{i=k-1} m+3i}= 
p^{\frac{3}{2}k^2+(m-\frac{3}{2})k}.$$

This finishes the proof of the theorem.

\section{Ridgeline numbers and small primes}
\label{rid_small}
We adopt the notations of  Definition~\ref{defn::v}.
%It is well-known 
We prove that $m-s=2(h^\vee -1)$ 
where $h^\vee$ is the dual Coxeter number of $\fg$
(see Proposition~\ref{Dual_Cox}). 
%(cf. \cite{kn::Wa}).
Therefore, 
$$
v(\fg)=\frac{l}{2(h^\vee -1)-r}.
$$
We present the values of $v(\fg)$ 
in Appendix~\ref{table1}. 
We include only Lie algebras in tolerable characteristics
(see Definition~\ref{defn::notsobad})
where our method produces new results.

Let us remind the reader that the very good characteristics are
$p\nmid l+1$ in type $A_l$,
$p\neq 2$ in types $B_l$, $C_l$, $D_l$,
$p\neq 2,3$ in types $E_6$, $E_7$, $F_4$, $G_2$,
and $p\neq 2,3,5$ in type $E_8$.
If $p$ is very good, the Lie algebra
$\fg$ behaves as in characteristic zero.
In particular, $\fg$ is simple, its Killing form
is non-degenerate, etc. 
Let us contemplate what calamities betide the Lie algebra $\fg$
in small characteristics.

Suppose that $p$ is tolerable but not very good.
If $p$ does not divide 
the determinant of the Cartan matrix of $\fg$, 
the Lie algebra $\fg$ is simple.
This covers the following primes:
$p= 2$ in types $E_6$ and $G_2$,
$p=3$ in types $E_7$ and $F_4$,
$p= 2,3,5$ in type $E_8$.
In this scenario, the $\fg$-modules
$\fg$ and $\fg^\ast$ are isomorphic, which 
immediately gives us
a non-degenerate invariant bilinear form on $\fg$ 
\cite[0.13]{kn::Hum}.
%We expect that this form is symmetric: it can probably be constructed 
%on Chevalley's $\bZ$-form $\fg_\bZ$ and then reduced to characteristic
%$p$.
%We don't use the symmetricity of the form, hence it is of no concern
%for us.

If $p$ divides the determinant of the Cartan matrix of $\fg$, there is
more than one Chevalley Lie algebra. 
We study the simply connected Lie algebra $\fg$, i.e., $[\fg,\fg]=\fg$
and $\fg/\fz$ is simple (where $\fz$ is the centre).
There is a canonical map to the adjoint Lie algebra $\fg^\flat$:
%(adjointness manifests itself in the fact that the centre
%of $\fg^\flat$ is zero):
$$
\varphi :
\fg = \fh \oplus \bigoplus_{\alpha} \fg_{\alpha}
\rightarrow 
\fg^\flat = \fh^\flat \oplus \bigoplus_{\alpha} \fg_{\alpha}. 
$$
The map $\varphi$ is the identity on the root spaces $\fg_\alpha$.
Let us describe it on the Cartan subalgebras.
The basis of the Cartan subalgebra $\fh$ are the 
simple coroots
$\vh_i=\alpha_i^\vee = [\ve_{\alpha_i}, \ve_{-\alpha_i}]$. 
The basis of the Cartan subalgebra $\fh^\flat$ are the fundamental coweights
$\vy_i=\varpi_i^\vee$ 
defined by
$\alpha_i (\vy_j)=\delta_{i,j}$.
Now the map $\varphi$ on the Cartan subalgebras
is given by
$$
\varphi (\vh_i) = \sum_j c_{j,i} \vy_j
$$
where $c_{j,i}$ are entries of the Cartan matrix of the coroot system
of
$\fg$.
The image of $\varphi$
is $[\fg^\flat,\fg^\flat]$.
The kernel of $\varphi$
is the centre~$\fz$. From our description
$\fz$ is the subspace of $\fh$ equal to the null space of the Cartan
matrix. It is equal to $\fg^0$, the kernel of $\eta$.
The dimension of $\fz$ is at most $2$
(see the values of $r$ in Appendix~\ref{table1}).

The key dichotomy now is whether the Lie algebra $\fg/\fz$ is simple or not.
If $\fg$ is simply-laced, the algebra $\fg/\fz$ is simple.
This occurs when 
$p\mid l+1$ in type $A_l$,
$p=2$ in types $D_l$ and $E_7$,
$p=3$ in type $E_6$.
Notice that 
$A_1$ in characteristic 2
needs to be excluded:
$\fg/\fz$ is abelian rather than simple.
In this scenario the
$\fg$-modules  
$\fg/\fz$
and
$(\fg/\fz)^\ast$
are isomorphic.
This gives us an invariant bilinear form with the kernel $\fz$
\cite[0.13]{kn::Hum}.

Let us look meticulously at 
$\fg$ of type $D_l$ when $p=2$. 
The standard representation gives a homomorphism of Lie algebras 
\begin{equation*}
\rho : \fg \rightarrow \fso_{2l} (\bK), \ \ 
\vx \mapsto \rho (\vx) =
\begin{pmatrix}
\rho_{11} (\vx) & \rho_{12} (\vx) \\
\rho_{21} (\vx) & \rho_{22} (\vx)
\end{pmatrix},
%, \ \ 
%\vy \mapsto \rho (\vy) =
%\begin{pmatrix}
%Y_{11} & Y_{12} \\
%Y_{21} & Y_{22}
%\end{pmatrix},
\end{equation*}
where $\rho_{22} (\vx)= \rho_{11}(\vx)^t$, while 
$\rho_{12}(\vx)$ and  $\rho_{21}(\vx)$ 
are skew-symmetric $l\times l$-matrices, 
which for $p=2$ is equivalent to symmetric with zeroes 
on the diagonal.
The Lie algebra $\fso_{2l} (\bK)$
has a 1-dimensional centre spanned by the identity matrix.
If $l$ is odd, $\rho$ is an isomorphism,
and 
$\fg$ has a 1-dimensional centre.
However, if $l$ is even, $\rho$ has a 1-dimensional kernel,
and 
$\fg$ has a 2-dimensional centre.
 
It is instructive to observe how the standard representation $\rho$
equips $\fg$ with an invariant form.
A skew-symmetric matrix $Z$ can  be written uniquely as a
sum $Z=Z^L + Z^U$, where $Z^L$ is strictly lower triangular and $Z^U$ is
strictly upper triangular. Then the bilinear form is given by 
$$
\eta (\vx,\vy) \coloneqq 
\langle \rho(\vx), \rho(\vy)\rangle \coloneqq  
\text{Tr} \, 
(\rho_{11} (\vx) \rho_{11} (\vy) + \rho_{12}(\vx)^L \rho_{21}(\vy)^U 
+ \rho_{21}(\vx)^L \rho_{12}(\vy)^U).
$$ 
This form $\eta$ is a reduction of
the form $\frac{1}{2}\text{Tr} \, (\varphi(\vx)\varphi(\vy))$ 
on $\fso_{2l} (\bZ)$, hence it is invariant.

Finally we suppose that $p$ is not tolerable.
This happens when 
$p=2$ in types $B_l$, $C_l$ and $F_4$
or
$p=3$ in type $G_2$.
In all these cases
$\fg$ is not simply-laced
and the quotient algebra $\fg/\fz$ is not simple.
The short root vectors generate a
proper non-central ideal $I$.
This ideal sits in
the kernel of any non-zero invariant form.
Consequently, our method fails to produce any new result.
%This happens when 
%$p=2$ in types $B_l$, $C_l$ and $F_4$
%or
%$p=3$ in type $G_2$. These are precisely the Lie algebras excluded
%by our ``tolerable'' condition.

\section{Proof of  Theorem~\ref{theorem::2}: the General Case} 

Let $\fa$ be an $m$-dimensional Lie algebra over a field $\bK$ 
of characteristic $p$ (prime or zero). We consider 
it as a topological
space in the Zariski topology.
We also consider a function 
$\dim \circ \fc:\fa \to \bR$ that for an element $\vx\in\fa$
computes the dimension of its centraliser $\fc(\vx)$.
\begin{lemma}
\label{lem1}
The function $\dim \circ \fc$
is upper semicontinuous, i.e.,  
for any number $n$
the set \newline
$\{\vx\in\fa\,\mid\, \dim (\fc (\vx))\leq n\}$ is Zariski open.
\end{lemma}
\begin{proof}
Observe that
$\fc (\vx)$ is the kernel of the adjoint operator $\text{ad}(\vx)$.
Thus, $\dim (\fc (\vx))\leq n$ is equivalent to 
$\mathrm{rank}(\mathrm{ad}(\vx))\geq m-n$. This is clearly an open condition,
given
by the non-vanishing of one of the $(m-n)$-minors.
\end{proof}

Now we move to 
$\overline{\bK}$, the 
algebraic closure of $\bK$.
Let $\bar{\fa}=\fa\otimes_\bK \overline{\bK}$.
To distinguish centralisers in $\fa$ and $\bar{\fa}$
we denote $\fc (\vx) : = \fc_\fa (\vx)$
and 
$\bar{\fc} (\vx) : = \fc_{\bar{\fa}} (\vx)$. 
Now we assume that $\bar{\fa}$ is the Lie algebra of a
connected algebraic group $\cA$. Let 
$\mathrm{Orb} (\vx)$ be the $\cA$-orbit of an element $\vx\in\bar{\fa}$.

\begin{lemma}
\label{lem2}
Let $\vx$ and  $\vy$ 
be elements of $\bar{\fa}$ 
such that $\vx \in \overline{\mathrm{Orb} (\vy)}$. Then 
$\dim \bar{\fc} (\vx) \geq \dim \bar{\fc} (\vy).$
\end{lemma}
\begin{proof} The orbit 
$\mathrm{Orb} (\vy)$ intersects any open neighbourhood of $\vx$,
and, in particular, the set 
$X=\{\vz\in\bar{\fa}\,\mid\, \dim (\bar{\fc} (\vz))\leq \dim (\bar{\fc} (\vx))\}$,
which is open by Lemma~\ref{lem1}.
If $\vz \in \mathrm{Orb} (\vy)\cap X$, then 
\newline
$\dim \bar{\fc} (\vx) \geq \dim \bar{\fc} (\vz) = \dim \bar{\fc} (\vy).$
\end{proof}

The stabiliser subscheme $\cA_\vx$ is, in general, non-reduced
in positive characteristic.
It is reduced (equivalently, smooth)
if and only if the inclusion
$\fc (\vx)\supseteq \mathrm{Lie} (\cA_\vx)$
is an equality (cf. \cite[1.10]{kn::Hum}).
If $\cA_\vx$ is smooth, the orbit-stabiliser
theorem implies that 
$$
\dim (\fa) = m = \dim \cA_\vx 
+ \dim \mathrm{Orb} (\vx)
= \dim \bar{\fc}(\vx) 
+ \dim \mathrm{Orb} (\vx).
$$
In particular, 
Lemma~\ref{lem2} follows from the inequality
$
\dim \mathrm{Orb} (\vx)
\leq \dim \mathrm{Orb} (\vy)
$.

Let us further assume that $\cA=\cG$ is 
a simple connected simply-connected
algebraic group %, $W$ is its Weyl group
and
$\bar{\fa}=\fg$
is a simply-connected Chevalley Lie algebra.
Let us fix a triangular decomposition
$\fg=\fn_-\oplus\fh\oplus\fn$.
An element $\vx\in\fg$ is called {\em semisimple} 
if $\Orb (\vx) \cap \fh \neq \emptyset$.
An element $\vx\in\fg$ is called {\em nilpotent} 
if $\Orb (\vx) \cap \fn \neq \emptyset$.
We call a representation
$\vx=\vx_\ts+\vx_\tn$
{\em a quasi-Jordan decomposition}
if 
$\vx_\ts\in g(\fh)$ (image of $\fh$ under $g$)  and 
$\vx_\tn\in g(\fn)$ for the same $g\in \cG$.

Recall that {\em a Jordan decomposition}
is a quasi-Jordan decomposition
$\vx=\vx_\ts+\vx_\tn$
such that 
$[\vx_\ts,\vx_\tn]=0$.
A Jordan decomposition exists and is unique 
if $\fg$ admits a non-degenerate
bilinear form
\cite[Theorem 4]{kn::KW}.

Notice that part (1) of the following lemma
cannot be proved by the argument that
the Lie subalgebra $\bK\vx$ is contained in a maximal soluble
subalgebra:
in characteristic $2$ the Borel subalgebra $\fb= \fh\oplus\fn$
is not maximal soluble.

%\begin{theorem}
%Suppose that $l\geq 2$. Suppose further\footnote{Here I am not yet
%  sure. I will sort it out.} that $p\neq 2$ or $\cG$ is not of type
%$C_l$.
%The following statements hold.
%\begin{enumerate}
%\item The restriction of functions gives an algebra isomorphism
%$\overline{\bK}[\bar{\fg}]^\cG \rightarrow \overline{\bK}[\fh]^W$.
%\item $\Orb (\vx)$ is closed if and only if $\vx$ is semisimple.
%\item $\fn$ is the set of nilpotent elements of the Borel subalgebra $\fb\coloneqq \fh\oplus\fn$.
%\item $\vx\in\bar{\fg}$ is nilpotent  if and only if $f(\vx)=f(0)$ for
%  all $f\in\overline{\bK}[\bar{\fg}]^\cG$.
%\item Every $\vx\in\bar{\fg}$ admits a unique Jordan decomposition
%  $\vx=\vx_\ts+\vx_\tn$.
%\item $\vx_\ts$ belongs to the orbit closure $\overline{\Orb(\vx)}$.
%\end{enumerate}
%\end{theorem}
\begin{lemma}
\label{qjordan}
Assume
that $p\neq 2$ or $\cG$ is not of type
$C_l$ (in particular, this excludes $C_2=B_2$ and $C_1=A_1$).
Then the following statements hold.
\begin{enumerate}
\item Every $\vx\in\fg$ admits a (non-unique) quasi-Jordan decomposition
  $\vx=\vx_\ts+\vx_\tn$.
\item $\vx_\ts$ belongs to the orbit closure $\overline{\Orb(\vx)}$.
\item If $\Orb (\vx)$ is closed,  then $\vx$ is semisimple.
\item $\dim \bar{\fc} (\vx_{s}) \geq \dim \bar{\fc} (\vx)$. 
\end{enumerate}
\end{lemma}
\begin{proof} (cf. \cite[Section 3]{kn::KW}.) (1) 
Our assumption on $\fg$ assures the existence of a regular
semisimple element $\vh\in\fh$, i.e., an element such
that $\bar{\fc} (\vh)=\fh$. 
The differential 
${\mathrm d}_{(e,\vh)}a : \fg \oplus \fh \rightarrow \fg$
of the action map
$a:\cG \times \fh \rightarrow \fg$ is given by
the formula
$$
{\mathrm d}_{(e,\vh)}a (\vx,\vk) = [\vx,\vh]+\vk . 
$$
Since the adjoint operator $\mathrm{ad}(\vh)$ 
is a diagonalizable operator whose
0-eigenspace is $\fh$, 
the kernel of ${\mathrm d}_{(e,\vx)}a$
is $\fh\oplus 0$.
Hence, the image of $a$ contains an open
subset of $\fg$.
Since the set $\cup_{g\in\cG} g(\fb)$ 
contains the image of $a$, it is a dense subset of $\fg$.

Let $\cB$ be the Borel subgroup of $\cG$
whose Lie algebra is $\fb$.
The quotient space $\cF=\cG/\cB$ 
is a flag variety. Since $\cF$ is projective,
the projection map $\pi: \fg \times \cF \rightarrow \fg$
is proper.
The Springer variety
$\cS=\{ (\vx,g(\cB)) \,|\, \vx\in g(\fb) \}$
is closed in $\fg \times \cF$.
Hence, $\cup_{g\in\cG} g(\fb) = \pi (\cS)$
is closed in $\fg$.
Thus, $\cup_{g\in\cG} g(\fb) = \fg$.
Choosing $g$ such that $\vx\in g(\fb)$ gives a decomposition.

(2) Suppose $\vx_\ts\in g(\fh)$. Let $\cT$ be the torus whose Lie
algebra is $g(\fh)$. We decompose $\vx$ over the roots of $\cT$:
$$
\vx = \vx_\ts + \vx_\tn 
= \vx_0 + \sum_{\alpha\in Y(\cT)} \vx_\alpha. 
$$
We can choose a basis of $Y(\cT)$
so that only positive roots appear. Hence, the action map
$a: \cT \rightarrow \fg$, 
$a(t) = t(\vx)$
extends alone the embedding $\cT\hookrightarrow \bK^l$
to a map
$\widehat{a}: \bK^l \rightarrow \fg$.
Observe that $\vx_\ts = \widehat{a}(0)$.

Let $U\ni \vx_\ts$ be an open subset of $\fg$.
Then ${\widehat{a}\,}^{-1}(U)$ is open in $\bK^l$
and $\cT\cap {\widehat{a}\,}^{-1}(U)$ is not empty.
Pick $t\in \cT\cap {\widehat{a}\,}^{-1}(U)$.
Then $a(t)=t(\vx)\in U$, thus, 
$\vx_\ts \in
\overline{\cT(\vx)}\subseteq 
\overline{\Orb(\vx)}$.

(3) This immediately follows from (1) and (2). 

(4) This immediately follows from (2) and Lemma~\ref{lem2}. 
\end{proof}

If $\alpha$ is a long simple root,
its root vector $\ve_{\alpha}\in{\fg}=\bar{\fa}$
is known as {\em the minimal nilpotent}.
The dimension of $\mathrm{Orb} (\ve_\alpha)$
is equal to  $2( h^\vee -1)$
%where $h^\vee$ is the dual Coxeter number
(cf. \cite{kn::Wa}).

\begin{prop}
\label{Dual_Cox}
Suppose that $l\geq 2$
and that the characteristic $p$ of the field $\bK$ is tolerable 
for
$\fg$.
Then for any noncentral $\vx\in\fa$ \ 
$$
\dim \fc (\vx) \leq \dim \bar{\fc} (\ve_\alpha) = m - 2( h^\vee -1).
$$
%The dimension of the orbit of a minimal nilpotent is minimal.
\end{prop}
\begin{proof} 
Let $\vx \in \fa$ ($\vy \in \fg$) be a noncentral element with $\fc
(\vx)$ 
($\bar{\fc} (\vy)$ correspondingly)  
of
the largest possible dimension. 
Observe that
$
\dim \fc (\vx) \leq \dim \bar{\fc} (\vx) \leq \dim \bar{\fc} (\vy)$.

Let us examine a quasi-Jordan decomposition $\vy=\vy_\ts+\vy_\tn$.
%We can decompose $\vx$ into its semisimple and
%nilpotent parts as $\vx = \vx_{ss} + \vx_{n}$. Then we have 
Since $\vy_{s}\in \overline{\text{Orb} (\vy)}$, 
we conclude that
$\dim \bar{\fc} (\vy_{s}) \geq \dim \bar{\fc} (\vy)$. But 
$\dim \bar{\fc}
(\vy)$ is assumed to be maximal. 
There are two ways to reconcile this:
either $\dim \bar{\fc} (\vy_{s}) = \dim \bar{\fc} (\vy)$,
or $\vy_\ts$ is central.
%Hence,  either $\vx_{s} = \vx$,
%i.e. $\vx$ is semisimple, or $\vx_{s} = 0$, i.e. $\vx$ is nilpotent. 

Suppose $\vy_\ts$ is central. Then $\vy$ and $\vy_\tn$
have the same centralisers. We may assume that
$\vy=\vy_\tn$ is nilpotent. Lemma~\ref{lem2}
allows us to assume
without loss of generality that the orbit 
${\text{Orb} (\vy)}$ is minimal, that is, 
$\overline{\text{Orb} (\vy)} = {\text{Orb} (\vy)}\cup\{0\}.$
On the other hand, the closure
$\overline{\text{Orb} (\vy)}$ contains a root vector $\ve_\beta$.

Let us prove the last statement.
First, observe that 
$\bK^\times \vy \subseteq \text{Orb} (\vy)$.
If $p$ is good, this immediately follows
from Premet's version of Jacobson-Morozov Theorem \cite{Pr}.
If $\text{Orb} (\lambda\vy)\neq \text{Orb} (\vy)$
in an exceptional Lie algebra in a bad tolerable characteristic,
then we observe two distinct nilpotent orbits 
with the same partition into Jordan blocks.
It never occurs: all the partitions are 
listed 
in the VIGRE paper
\cite[section 6]{kn::VIGRE}.
The remaining case of $p=2$ and $\fg$ is of type $D_l$
is also settled in the  VIGRE paper
\cite{kn::VIGRE}.
Now let $\vy\in g(\fn)$, and  $\cT_0$ be the torus whose Lie
algebra is $g(\fh)$. 
Consider $\cT\coloneqq \cT_0 \times \bK^\times$
with the second factor acting on $\fg$ via the vector space structure.
Write $\vy=\sum_{\beta\in Y(\cT_0)} \vy_\beta$
using the roots of $\cT_0$. The closure 
of the orbit $\overline{\cT(\vy)}$ is 
contained in $\overline{\text{Orb} (\vy)}$.
Let us show that $\overline{\cT(\vy)}$ 
contains one of $\vy_\beta$. 
Let us write $\cT_0=G_\tm\times G_\tm \times \ldots \times G_\tm$ and decompose
$\vy=\vy_k+\vy_{k+1}+\ldots + \vy_n$ using the weights of the first
factor $G_\tm$ with $\vy_k\neq 0$. Then
$$
\cT(\vy) \supseteq
\{ (\lambda,1,1\ldots,1,\lambda^{-k})\cdot \vy| \lambda \in \bK^\times \} =
\{ \vy_k+\lambda^1 \vy_{k+1}+\ldots + \lambda^{n-k}\vy_n | \lambda \in \bK^\times \}.
$$  
Hence, $\vy_k \in \overline{\cT(\vy)}$.
Repeat this argument with $\vy_k$ instead of $\vy$
for the second factor of $\cT_0$, and so on.
At the end we arrive at nonzero $\vy_\beta$, hence, 
$\ve_\beta\in \overline{\text{Orb} (\vy)}$. 

Without loss of
generality we now assume that $\vy=\ve_\beta$ for a simple root
$\beta$. 
If $p$ is good, then $\dim(\bar{\fc} (\ve_\beta))$ does not depend on the
field:
$$
\bar{\fc} (\ve_\beta) =
\ker (d\beta : \fh \rightarrow \bK) \oplus 
\bigoplus_{\gamma+\beta \mbox{ is not a root}} \fg_{\gamma}.
$$
In particular, it is as in characteristic zero: the long root vector
has a larger centraliser then the short root vector and
%all the stabilisers are smooth and 
%$\vy$ belongs to $\Orb (\ve_\alpha )$. 
%must be the minimal nilpotent.
%Hence,
$
\dim \bar{\fc} (\vy) = \dim \bar{\fc} (\ve_\alpha) = m - 2( h^\vee -1)
$
\cite{kn::Wa}.
If $p=2$ and $\fg$ is of type $D_l$,
then %$\vy$ is also a minimal nilpotent.
a direct calculation gives 
the same formula for $\dim \bar{\fc} (\ve_\alpha)$.
%the correct dimension of the centraliser.
In the exceptional cases in bad characteristics the orbits and 
their centralisers are computed in the VIGRE paper
\cite{kn::VIGRE}.
One goes through their tables and establishes the formula
for $\dim \bar{\fc} (\vy)$ in all the cases.

Now suppose $\dim \bar{\fc} (\vy_{s}) = \dim \bar{\fc} (\vy)$.
We may assume that $\vy=\vy_\ts$ is semisimple.
Then $\vy$ is in some Cartan subalgebra $g^{-1}(\fh)$ and
$\dim \bar{\fc} (g(\vy)) = \dim \bar{\fc} (\vy)$.
Moreover, 
$$\bar{\fc} (g(\vy)) = \fh \oplus \bigoplus_{ \{ \alpha | \alpha(g(\vy)) = 0 \} }
\fg_{\alpha} 
$$ 
is a reductive subalgebra.
If $\varphi :\fg 
\rightarrow 
\fg^\flat$
is the canonical map (see Section~\ref{rid_small}),
then $\dim \bar{\fc} (g(\vy)) = \dim \fc_{\fg^\flat} (\varphi(g(\vy)))$.
It remains 
to examine the Lie algebras case by case 
and exhibit a non-zero element in $\fh^\flat$ with the maximal 
dimension of centraliser. This is done in Appendix~\ref{table1}.
\end{proof}

Now we can give a crucial dimension estimate for the proof of Theorem~\ref{theorem::2}.
 
\begin{prop} 
\label{my_estimate}
Let $\fa$ be an $m$-dimensional Lie algebra
with an associative bilinear form
$\eta$,
whose kernel $\fa^0$ is the centre of $\fa$.
Suppose $r=\dim(\fa^0)$ and 
$k\geq\dim (\fc (\vx))$ 
for
any non-central element $\vx\in\fa$.
Finally, let $U, V$ be subspaces of $\fa$ such that
$\dim (U)+\dim (V) > m+k+r$. Then
$[U,V] = \fa.$
\end{prop}
\begin{proof}
Suppose not. 
Let us consider the orthogonal complement
$W= [U,V]^\perp \neq \fa^0$
under the form $\eta$.
Observe that $U\subseteq [V,W]^\perp$ since $\eta$ is associative.
But $W$ admits a noncentral element $\vx\in W$ so that 
$\dim (\fc(\vx))\leq k$. Hence
$$
\dim ([V,W]) \geq \dim (V) -k
\ \mbox{ and } \ 
\dim ([V,W]^\perp) \leq m+k+r - \dim (V).
$$
Inevitably,
$
\dim (U) \leq m+k+r - \dim (V).
$
\end{proof}

We may now prove the first part of Theorem~\ref{theorem::2}.
%\newline
%{\em Proof of Theorem~\ref{theorem::2}.}
%The dimension of the the centraliser of the minimal element is
%$n-(2h^\vee -2)$.
%Thus, if $\dim (U) + \dim (V) > 2n-(2h^\vee -2)$, we are done:
We use $m$, $l$, $r$ and $s$ as in
Definition~\ref{defn::v}.
If $\dim (U) + \dim (V) > m+s+r$, 
we are done by Proposition~\ref{my_estimate}:
$$
\cod ([U,V]) = 0 \leq (1+v(\fg))(\cod (U) + \cod (V)).
$$
%Now we assume that $\dim (U) + \dim (V) \leq 2n-(2h^\vee -2)$.
Now we assume that $\dim (U) + \dim (V) \leq m+s+r$.
It is known \cite{kn::ANS}
that
$$
\cod ([U,V]) \leq l + \cod (U) + \cod (V).
$$
It remains to notice that
$
l = v(\fg) (m-s-r) \leq v (\fg) (\cod (U) + \cod (V)).
%l = v (2h^\vee -2) \leq v (\cod (U) + \cod (V)).
$
The theorem is proved.

\section{Proof of Theorem~\ref{theorem::2}: Rank 2} 
In this section $\mathbf{G}$ is a Chevalley group scheme of rank $2$.
%Let $\bK$ be a field of characteristic $0$ or $p$ and $G=\mathbf{G}(\bK)$ be a simple Chevalley group over $\bK$ of rank $2$.
%As before $\fg$ denotes its Lie algebra with $m\coloneqq \dim\fg$.
%We further assume that the characteristic $p$ is very good for $\fg$, that is
%we avoid $p=2$ if $G$ is of types $C_2$ and $G_2$,
%and  $p=3$ if $G$ is of types $A_2$ and $G_2$.
%Let $\bK$ be a field of characteristic $0$ or $p$ and $G=\mathbf{G}(\bK)$ be a simple Chevalley group over $\bK$.
%Suppose now that the rank of $G$ is $2$ and let $\fg$ be its Lie algebra.
%Denote by $n$ the dimension of $\fg$. 
The characteristic $p$ of the field $\bK$ is zero or very good for $\fg$.
%we avoid $p=2$ if $G$ is of types $B_2$, $G_2$,
%and 
%$p=3$ if $G$ is of types $A_2$, $G_2$.
Let $\{\alpha, \beta\}$ be the set of simple roots of $\fg$ with $|\beta|\leq|\alpha|$. 
If $\fg$ is of type $A_2$ then $\alpha$ and $\beta$ have the same length. 
The group $\cG=\mathbf{G}(\overline{\bK})$ acts on on $\fg$ via the adjoint
action. 
By $\fc (\vx)$ we denote the centraliser $\fc_\fg (\vx)$
in this section. 
Let us summarise some standard facts about this adjoint action (cf.
\cite{kn::Hum}).  
\begin{enumerate}
\item If $\vx \in \fg$, the stabiliser $\cG_\vx$ is smooth, i.e., its Lie algebra
is the centraliser $\fc (\vx)$.
\item The dimensions  
$ \dim (\Orb (\vx))= \dim (\cG) -\dim (\fc  (\vx))$ and $\dim (\fc  (\vx))$ are even.
\item If $\vx\neq 0$ is semisimple, $\dim (\fc(\vx))\in \{ 2,4\}$.
Hence, $\dim (\Orb(\vx))\in \{ m-2,m-4\}$.
\item 
%Suppose  that $\vx \in \fg$ and $\vx =\vx_\ts +\vx_\tn$ is its Jordan
%decomposition.
%Then $G\vx_\ts \subseteq \overline{G \vx}$. In particular, 
%$\dim (c(\vx_\ts)) > \dim  (c(\vx))$
%unless $\vx=\vx_\ts$.
%Since a non-zero semisimple element $\vx_\ts$ has 
%$\dim (c(\vx_\ts))\in \{ 2,4\}$,
A truly mixed element $\vx =\vx_\ts +\vx_\tn$ 
(with non-zero semisimple and nilpotent parts)
is regular, i.e., $\dim (\fc (\vx))=2$ (cf. Lemma~\ref{qjordan}).
\item $\vx$ is nilpotent if and only if $\overline{\Orb(\vx)}$
  contains $0$.
\item There is a unique orbit of regular nilpotent elements
$\Orb(\ve_{r})$ where $\ve_{r} = \ve_{\alpha} + \ve_\beta$.
In particular, $\dim (\fc(\ve_r))=2$ and $\dim (\Orb(\ve_r))=m-2$.
\item For two nilpotent elements $\vx$ and $\vy$ we write $\vx\succeq\vy$
if $\overline{\Orb(\vx)}\supseteq \Orb(\vy)$. The following are
representatives of all the
nilpotent orbits in $\fg$ (in brackets we report $[\dim (\Orb(\vx)),\dim (\fc (\vx)) ]$):
 
      \begin{enumerate}
      \item If $\mathbf{G}$ is of type $A_2$, then
$$
\ve_r \,[6,2]\,
\succeq
\ve_\alpha  \,[4,4]\,
\succeq
0 \,[0,8]\,
.$$
           \item If $\mathbf{G}$ is of type $C_2$, then $\ve_\alpha$ and $\ve_\beta$ are no longer in the same orbit and so we have
$$
\ve_r \,[8,2]\,
\succeq
\ve_\beta  \,[6,4]\,
\succeq
\ve_\alpha  \,[4,6]\,
\succeq
0 \,[0,10]\,
.$$ 
            \item If $\mathbf{G}$ is of type $G_2$, there is an additional subregular nilpotent orbit
of an element $\ve_{sr} = \ve_{2\alpha +3\beta} + \ve_\beta$. 
In this case we have
$$
\ve_r \,[12,2]\,
\succeq
\ve_{sr} \,[10,4]\,
\succeq
\ve_\beta  \,[8,6]\,
\succeq
\ve_\alpha  \,[6,8]\,
\succeq
0 \,[0,14]\,
.$$
               \end{enumerate}
\end{enumerate}

We will now prove Theorem~\ref{theorem::2} for groups 
of type $A_2$, $C_2$ and $G_2$. 
We need to show that  if $U$ and  $V$  are subspaces of $\fg$, 
then
\begin{equation}
%\tag{$\spadesuit$}
\label{spadesuit}
\dim ([U,V]) \geq \dim (U) + \dim (V) - \dim\fg.
\end{equation}
We will use the following device 
repeatedly:
\begin{lemma}
\label{lemX}
The inequality
\begin{equation}
%\tag{$\clubsuit$}
\label{clubsuit}
\dim ([U,V]) \geq  \dim (V) - \dim (V\cap \fc(\vx))
\end{equation}
holds for any $\vx\in U$. 
In particular, if there exists $\vx\in U$ such that
$
\dim (U) + \dim (V\cap \fc(\vx)) \leq \dim \fg,
$
then inequality~(\ref{spadesuit}) holds.
\end{lemma}
\begin{proof}
It immediately follows from the observation
%\begin{equation*}
%\tag{$\clubsuit$}
$[U,V] \supseteq [\vx,V] \cong {V}\Big/  ({V\cap \fc(\vx)})
$. 
%\end{equation*}
\end{proof}

Now we give a case-by-case proof of  inequality~(\ref{spadesuit}). 
Without loss of generality we assume that $1\leq \dim(U)\leq\dim (V)$
and that the field $\bK$ is algebraically closed.

 \subsection{$\mathbf{G}=A_2$} 
Using the standard facts,
observe that if $\vx\in\fg\setminus\{0\}$, then $\dim(\fc(\vx))\leq
4$. 
Moreover, if $\dim(\fc(\vx))=4$, then either $\vx\in \Orb(\ve_{\alpha})$, or $\vx$ is semisimple.
 Since $\dim\fg=8$,
 we need to establish that
 $$
\dim ([U,V]) \geq \dim (U) + \dim (V) - 8
$$
\vskip 3mm
Now we consider various possibilities.

{\bf Case 1:} If $\dim(U)\leq 4$, then
$
\dim (V\cap \fc(\vx)) \leq \dim (\fc(\vx)) \leq 4 \leq 8 - \dim (U)
$
for any nonzero $\vx\in U$.
We are done by Lemma~\ref{lemX}.
%Using $(\clubsuit)$, 
%$$
%\dim ([U,V]) \geq \dim (V) - \dim (V\cap \fc(\vx)) \geq \dim (U) + \dim (V) - 8.
%$$

{\bf Case 2:} If $\dim(U)+\dim(V)> 12$, then the hypotheses of
Proposition~\ref{my_estimate} hold. Hence, $[U,V]=\fg$ that obviously implies the desired conclusion.
\vskip 3mm
Therefore we may  suppose that $\dim(U)+\dim(V)\leq 12$ and $\dim U\geq 5$. This leaves us with the following two cases.

{\bf Case 3:}  $\dim(U)=5$ and $\dim(V)\leq 7$. We need to show that
$$
\dim ([U,V]) \geq \dim (U) + \dim (V) - 8 = \dim (V) - 3.
$$
As $\dim(\overline{\Orb(\ve_{\alpha})})=4$,
we may pick $\vx \in U$ with $\vx \not\in \overline{\Orb(\ve_\alpha)}$.
If $\vx$ is regular,
we are done by Lemma~\ref{lemX} 
since $\dim (V\cap \fc(\vx)) \leq \dim (\fc(\vx)) = 2$.
%then $(\clubsuit)$ immediately gives a desired estimate 
%$\dim([\vx,V])\geq\dim (V)-2$.
If $\vx$ is not regular, then $\dim(\fc(\vx))=4$ and $\vx$ is
semisimple. 
In particular, 
its centraliser $\fc(\vx)$ contains a Cartan subalgebra $g(\fh)$ of
$\fg$.

Let us consider the intersection $V\cap \fc(\vx)$.
If
$\dim (V\cap \fc(\vx)) \leq 3$, 
we are done by Lemma~\ref{lemX}.
%, for using $(\clubsuit)$ we have  
%$$
%\dim ([U,V]) \geq \dim (V) - \dim (V\cap \fc(\vx)) \geq \dim (V)-3.
%$$
Otherwise, $V\supseteq \fc(\vx)$ and $V$ contains a regular 
semisimple element 
$\vy\in g(\fh)\subseteq V$.
If $U\supseteq \fc(\vy)=g(\fh)$, then $U\ni \vy$
and we are done by Lemma~\ref{lemX} as in the previous paragraph.
%We may use $(\clubsuit)$ with $\vy$ instead of $\vx$:
%$$
%\dim ([U,V]) \geq \dim (V) - \dim (\fc(\vy)) \geq \dim (V)-2.
%$$
Otherwise, $\dim (U\cap \fc(\vy)) \leq 1$
and
we finish the proof using  Lemma~\ref{lemX}:
$$
\dim ([U,V]) \geq \dim (U) - \dim (U\cap \fc(\vy)) \geq 5-1 = 4 \geq \dim (V)-3.
$$

{\bf Case 4:}  $\dim(U)=\dim(V)= 6$.  
This time we  must show that 
$$
\dim([U,V])\geq 4 = \dim (V) -2.
$$
By Lemma~\ref{lemX} it suffices to find a regular element in $\vx \in
U$ (or in $V$) since 
$\dim (V\cap \fc(\vx)) \leq \dim (\fc(\vx)) = 2$.
Observe that
$$
\dim (U\cap V) \geq \dim (U) + \dim (V) - 8 = 4 =
\dim (\overline{\Orb(\ve_\alpha)}) .
$$
Since $\overline{\Orb(\ve_\alpha)}$ is an 
irreducible algebraic variety and not an affine space,
there exists $\vx \in U\cap V$ such that
$\vx\not\in\overline{\Orb(\ve_\alpha)}$.
If $\vx$ is regular, 
we are done.
%by Lemma~\ref{lemX}
%since $\dim (V\cap \fc(\vx)) \leq \dim (\fc(\vx)) = 2$.
%$$
%\dim ([U,V]) \geq \dim (V) - \dim (V\cap \fc(\vx)) \geq 6-2 
%=4 .
%$$
If $\vx$ is not regular,  $\vx$ is semisimple and 
its centraliser $\fc(\vx) = \bK\vx \oplus \fl$,  a direct sum of Lie algebras
$\bK\vx\cong \bK$ and $\fl\cong\fsl_2(\bK)$.

Consider the intersection $V\cap \fc(\vx)$. 
If $\dim (V\cap \fc(\vx)) \leq 2$, 
we are done by Lemma~\ref{lemX} as before.
%applying $(\clubsuit)$ we obtain
%$$
%\dim ([U,V]) \geq \dim (V) - \dim (V\cap \fc(\vx)) \geq 6-2=4.
%$$
Assume that $\dim (V\cap \fc(\vx)) \geq 3$.
If $\dim (V\cap \fc(\vx))=4$, $V$ contains $\fc(\vx)$
and consequently a regular semisimple
element $\vy$. 
%We finish the proof using $(\clubsuit)$:
%$$
%\dim ([U,V]) \geq \dim (U) - \dim (U\cap \fc(\vy)) \geq 6-2=4.
%$$

Finally, consider the case $\dim (V\cap \fc(\vx))=3$. 
Let $\pi_2$ be the natural projection $\pi_2:\fc(\vx)\rightarrow\fl$ and set
$W\coloneqq  \pi_2 (V\cap \fc(\vx))$. Since $\bK\vx\subseteq V\cap \fc(\vx)$,
the subspace $W$ of $\fsl_2({\bK})$
is 2-dimensional. Clearly, 
$V\cap \fc(\vx)\subseteq{\bK}\vx \oplus W$.
Since both spaces have dimension 3,
$V\cap \fc(\vx) = {\bK}\vx \oplus W$.
Then $W=\va^\perp$ 
(with respect to the Killing form), 
where $0\neq \va\in\fsl_2({\bK})$ is  either semisimple or
nilpotent.
In both cases $W$ contains a nonzero nilpotent  element $\vz$.
Thus, we have found a regular element 
$\vx+\vz\in V\cap \fc(\vx)$. 
%We may now use $(\clubsuit)$ replacing $\vx$ by $\vx+\vz$.
This finishes the proof for $A_2$.

\subsection{$\mathbf{G}=C_2$} Notice that this time $\dim (\fc (\vx))\leq 6$ for all $0\neq\vx\in\fg$. 
Moreover,  if $\dim (\fc (\vx))=6$, $\vx\in \Orb(\ve_\alpha)$. Finally,
the set $\overline{\Orb(\ve_\alpha)}= \Orb(\ve_\alpha) \cup \{0\}$ is a 4-dimensional cone,
%thus does not contain a 5-dimensional vector space. 
and the set $\overline{\Orb(\ve_\beta)} =  \Orb(\ve_\beta) \cup
\Orb(\ve_\alpha) \cup \{0\}$ is a 6-dimensional cone.
  %, so it does not contain a 7-dimensional vector space. 

As $\dim\fg=10$, this time 
we need to show that
$$\dim ([U,V])\geq \dim(U)+\dim(V)-10 = \dim(V)- (10 - \dim(U)).$$

{\bf Case 1:} $\dim(U)\leq 4$. 
We are done by Lemma~\ref{lemX} since 
for any  $0\neq \vx \in U$,
$$
\dim (V\cap \fc(\vx)) \leq \dim (\fc(\vx)) \leq 6 \leq 10 - \dim
(U).$$ 
%Hence, by $(\clubsuit)$
%$$
%\dim ([U,V]) \geq \dim (V) - \dim (V\cap \fc(\vx)) \geq \dim (U) + \dim (V) - 10.
%$$

{\bf Case 2:} $5\leq \dim(U)\leq 6$.  Hence, 
we may choose $\vx \in U$ such that $\vx \not\in \overline{\Orb(\ve_\alpha)}$.
We are done by Lemma~\ref{lemX} since
$$
\dim (V\cap \fc(\vx)) \leq \dim (\fc(\vx)) \leq 4 \leq 10 - \dim (U)
.$$ 
%which combined with $(\clubsuit)$ gives us
%$$
%\dim ([U,V]) \geq \dim (V) - \dim (V\cap \fc(\vx)) \geq \dim (U) + \dim (V) - 10.
%$$

{\bf Case 3:} If $\dim(U)+\dim(V)> 16$, then
 then the hypotheses of Proposition~\ref{my_estimate} hold. Hence, $[U,V]=\fg$, which implies the desired conclusion.

Therefore, we may  assume that $\dim(U)+\dim(V)\leq 16$ and 
$\dim (U)\geq 7$. This leaves us with the remaining two cases.

{\bf Case 4:}  $\dim(U)=7$, $\dim(V)\leq 9$. 
Now we must show that $\dim ([U,V]) \geq  \dim (V) - 3.$
By Lemma~\ref{lemX} it suffices to pick $\vx \in U$ 
with 
$\dim (V\cap \fc(\vx)) \leq 3$.
In particular,  a regular element will do.

Let us choose $\vx \in U$ such that $\vx \not\in \overline{\Orb(\ve_\beta)}$.
If $\vx$ is regular, we are done. % by using $(\clubsuit)$.
If $\vx$ is not regular,  $\vx$ is semisimple.
Hence, 
its centraliser $\fc(\vx)$ contains a Cartan subalgebra $g(\fh)$.
Let us consider the intersection $V\cap \fc(\vx)$.
If
$\dim (V\cap \fc(\vx)) \leq 3$, we are done again.
%as  
%$$
%\dim ([U,V]) \geq \dim (V) - \dim (V\cap \fc(\vx)) \geq \dim (V)-3.
%$$
Assume that $\dim (V\cap \fc(\vx)) =4$. Consequently, $V\supseteq \fc(\vx)$ and $V$ contains a regular 
semisimple element
$\vy\in g(\fh)\subseteq V$.
Now if $U\supseteq \fc(\vy)=g(\fh)$, then 
we have found a regular element
$\vy\in U$.
% and we may use $(\clubsuit)$: 
%$$\dim ([U,V]) \geq \dim (V) - \dim (\fc(\vy)) \geq \dim (V)-2>\dim(V)-3.
%$$
Otherwise, $\dim (U\cap \fc(\vy)) \leq 1$, 
and so, as $\vy\in V$, we finish using inequality~(\ref{clubsuit})   
of Lemma~\ref{lemX}:
$$
\dim ([U,V]) \geq \dim (U) - \dim (U\cap \fc(\vy)) \geq 7-1 = 6 \geq \dim (V)-3.
$$

{\bf Case 5:}  $\dim(U)=\dim(V)= 8$.
Let us observe that
$$
\dim (U\cap V) \geq \dim (U) + \dim (V) - 10 = 6 =
\dim (\overline{\Orb(\ve_\beta)}).
$$
Since $\overline{\Orb(\ve_\beta)}$ is an 
irreducible algebraic variety and not an affine space,
there exists $\vx \in U\cap V$ such that
$\vx\not\in\overline{\Orb(\ve_\beta)}$.
If $\vx$ is regular, we are done  by Lemma~\ref{lemX}:
$$
\dim ([U,V]) \geq \dim (V) - \dim (V\cap \fc(\vx)) \geq 8-2 
=6=
\dim (U) + \dim (V) - 10.
$$
If $\vx$ is not regular,  then $\vx$ is semisimple and 
its centraliser $\fc(\vx) = \bK\vx \oplus \fl$, a direct sum of Lie algebras
$ {\bK}$ and $\fl\cong\fsl_2({\bK})$.
If $\dim (V\cap \fc(\vx)) \leq 2$, then by Lemma~\ref{lemX} 
$$
\dim ([U,V]) \geq \dim (V) - \dim (V\cap \fc(\vx)) \geq 8-2=6.
$$
Thus we may assume that $\dim (V\cap \fc(\vx)) \geq 3$.
We now repeat the  argument from the last paragraph of $\S$ 4.1. This concludes $\S$ 4.2.

\subsection{$\mathbf{G}=G_2$} In this case $\dim (\fc (\vx))\leq 8$ for all $0\neq\vx\in\fg$. Moreover, 
if $\dim (\fc (\vx))=8$, then $\vx\in \Orb(\ve_\alpha)$. The centre of $\fc(\ve_\alpha)$ is $\overline{\bK}\ve_\alpha$.
Finally, the set $\overline{\Orb(\ve_\alpha)} = \Orb(\ve_\alpha) \cup \{0\}$ is a 6-dimensional cone,
the set $\overline{\Orb(\ve_\beta)} = \Orb(\ve_\beta) \cup
\Orb(\ve_\alpha) \cup \{0\}$ 
is an 8-dimensional cone 
and 
the set $\overline{\Orb(\ve_{sr})} = \Orb(\ve_{sr}) \cup
\Orb(\ve_\beta) 
\cup \Orb(\ve_\alpha) \cup \{0\}$ is a 10-dimensional cone.

As  $\dim \fg =14$, our goal now is to show that
$$
\dim ([U,V]) \geq \dim (U) + \dim (V) - 14
$$
In order to do so, as before, we are going to consider several mutually exclusive cases.

{\bf Case 1:} $\dim(U)\leq 6$.
We are done by Lemma~\ref{lemX} since 
for any  $0\neq \vx \in U$,
%For any non-zero $\vx \in U$, 
$$
\dim (V\cap \fc(\vx)) \leq \dim (\fc(\vx)) \leq 8 \leq 14 - \dim (U)
.$$ 
%Hence, by $(\clubsuit)$:
%$$
%\dim ([U,V]) \geq \dim (V) - \dim (V\cap \fc(\vx)) \geq \dim (U) + \dim (V) - 14.
%$$

{\bf Case 2:} $7\leq \dim(U)\leq 8$.
In this case we may choose $\vx \in U$ such that 
$\vx \not\in
\overline{\Orb (  \ve_\alpha )}$. 
We are done by Lemma~\ref{lemX} since
$$
\dim (V\cap \fc(\vx)) \leq \dim (\fc(\vx)) \leq 6 \leq 14 - \dim (U)
.$$
%, using $(\clubsuit)$, we obtain that
%$$
%\dim ([U,V]) \geq \dim (V) - \dim (V\cap \fc(\vx)) \geq \dim (U) + \dim (V) - 14.
%$$

{\bf Case 3:} $9\leq \dim(U)\leq 10$.
Now we may pick $\vx \in U$ such that $\vx \not\in \overline{\Orb (\ve_\beta)}$. 
Again we are done by Lemma~\ref{lemX} since
%Hence, 
$$
\dim (V\cap \fc(\vx)) \leq \dim (\fc(\vx)) \leq 4 \leq 14 - \dim (U)
.$$
%Using $(\clubsuit)$, we arrive  at the same conclusion:
%$$
%\dim ([U,V]) \geq \dim (V) - \dim (V\cap \fc(\vx)) \geq \dim (U) + \dim (V) - 14.
%$$

{\bf Case 4:}  If $\dim(U)+\dim(V)> 22$, then
$[U,V]=\fg$ by Proposition~\ref{my_estimate}.  
%This gives the desired result.
This leaves us with a single last possibility.

{\bf Case 5:}  $\dim(U)=\dim(V)= 11$. It remains to show that
$$
\dim ([U,V]) \geq 8 = \dim (V) -3.
$$
By dimension considerations we can choose $\vx \in U$ such that 
$\vx \not\in \overline{\Orb (\ve_{sr})}$.
Then $\dim (\fc(\vx)) \leq 4$. 
If $\dim (V\cap \fc(\vx)) \leq 3$, we are done by by Lemma~\ref{lemX}.
%$(\clubsuit)$:
%$$
%\dim ([U,V]) \geq \dim (V) - \dim (V\cap \fc(\vx)) \geq 11-3 =8.
%$$
Thus we may assume  that $\dim (\fc(\vx)) =4$ and $\fc(\vx)\subseteq V$.
Since $\vx$ is not nilpotent, 
$\vx$ must be semisimple. Hence, $\fc(\vx)\subseteq V$
contains a Cartan subalgebra $g(\fh)$ and, 
therefore, a regular semisimple element $\vy\in g(\fh)$.
We are done by Lemma~\ref{lemX}:
$$
\dim (U\cap \fc(\vy)) \leq \dim (\fc(\vy)) \leq 2.
$$ 
%Thus, by $(\clubsuit)$: 
%$$
%\dim ([U,V]) \geq \dim ([U,\vy]) \geq 
%\dim (U) - \dim (U\cap \fc(\vy)) 
%\geq 
%\dim (U) - \dim (\fc(\vy)) 
%\geq 11-2 > 8.
%$$
We have finished the proof of Theorem~\ref{theorem::2}.

%\section{$C_l$ in characteristic 2}
%{\bf It is a sketch that demonstrates that we have no hope in this case:-(}
%The dimension of the commuting variety is $n+3l$ in this case.
%Consequently, Abert++ method gives estimates 
%$$
%\dim ([U,V]) \leq \dim (U) + \dim (V) - n - 3l
%\ \ \mbox{ and } \ \ 
%\cod ([U,V]) \leq 3l + \cod (U) + \cod (V).
%$$
%The form on $C_l$ is the trace form $\langle X,Y\rangle = \mbox{Tr} (XY)$.
%Its kernel has dimension $l^2$. Thus, we can conclude that $[U,V]=\fg$
%as soon as
%$$
%\dim (U) +\dim (V) > n+k+r = 2n - (n-k) + r = 2n - 2(h^\vee -1) +r = 2(2l^2+l) - 2l +l^2=5l^2. 
%$$
%That is not very useful because 
%$5l^2 > 2(2l^2+2l)= 2 \dim
%(\fg)$,
%if $l>4$. 
%This agrees with the ridgeline number being negative. 
%
%Thus, unless we can find a better form (with much smaller nullity), we
%have nothing to say in this case...

\appendix
\section{Ridgeline numbers and maximal dimensions of centralisers}
\label{table1}
Column 3 contains the nullity $r$
of an invariant form. It is equal to
$\dim \fz$. 
Column 4 contains the dual Coxeter number.
Column 5 contains the ridgeline number.
Column 6 contains dimension of the centraliser of the minimal nilpotent.
Column 7 contains a minimal non-central semisimple element in $\fg^\flat$,
using simple coweights $\vy_i$
and the enumeration of roots in Bourbaki \cite{kn::Bo}.
Column 8 contains the dimension of the centraliser of the minimal semisimple
element
in $\fg^\flat$.
\begin{center}
%\caption{}
%\vskip 3mm
\renewcommand{\arraystretch}{2}
\begin{tabular}{|c|c|c|c|c|c|c|c|} \hline
type of $\fg$ &  $p$ & $r$ & $h^{\vee}$ &  $v(\fg)$  &   $m - 2( h^\vee -1)$ & $\vy$ &
$\dim \fc(\vy)$\\
\hline
$A_l,\ l\geq 2 $ & $(p,l+1)=1$ &  $0$ & $l+1$ & $\frac{1}{2}$ & $l^2$ & $\vy_1$ & $l^2$\\
$A_l, \ l\geq 2 $ & $p\mid(l+1)$ &  $1$ & $l+1$ & $\frac{l}{2l-1}$ & $l^2$ & $\vy_1$ & $l^2$\\
$B_l, \ l\geqslant 3$ &$p\neq 2$ & $0$ & $2l-1$   &   $\frac{1}{4}(1+\frac{1}{l-1})$  & $2l^2-3l+4$ & $\vy_1$ &    $2l^2-3l+2$ \\
$C_l, \ l\geq 2$ &$p\neq 2$ &  $0$ & $l+1$ &  $\frac{1}{2}$ &  $2l^2-l$ & $\vy_1$ &  $2l^2-3l+2$  \\
$D_l,\  l\geq 4$ &$p\neq 2$ &  $0$ & $2l-2$ & $\frac{1}{4}(1+\frac{3}{2l-3})$ & $2l^2-5l+6$ &  $\vy_1$ & $2l^2-5l+4$\\
$D_l,\  l=2l_0\geq 4$ & $2$ &  $2$ & $2l-2$ & $\frac{1}{4}(1+\frac{2}{l-2})$ & $2l^2-5l+6$ &  $\vy_1$ & $2l^2-5l+4$\\
$D_l,\  l=2l_0+1\geq 4$ & $2$ &  $1$ & $2l-2$ &
$\frac{1}{4}(1+\frac{7}{4l-7})$ & $2l^2-5l+6$ &  $\vy_1$ & $2l^2-5l+4$\\
$G_2$ & $p> 3$ & $0$ & $4$ & $\frac{1}{3}$ & $8$ & $\vy_1$ & $4$\\
$G_2$ & $p=2$ & $0$ & $4$ & $\frac{1}{3}$ & $8$ & $\vy_1$ & $6$ \\
$F_4$ & $p\neq 2$ &  $0$ & $9$& $\frac{1}{4}$ & $36$ & $\vy_1$ & $22$\\
$E_6$ & $p\neq 3$ &  $0$ & $12$ & $\frac{3}{11}$ & $56$ & $\vy_1$ & $46$\\
$E_6$ & $3$ &  $1$ & $12$ & $\frac{2}{7}$ & $56$ & $\vy_1$ & $46$\\
$E_7$ & $p\neq 2$ &  $0$ & $18$ &  $\frac{7}{34}$ & $99$ & $\vy_7$ & $79$ \\
$E_7$ & $p=2$ &  $1$ & $18$ &  $\frac{7}{33}$ & $99$ & $\vy_7$ & $79$ \\
$E_8$ &  $p\neq 2$ & $0$ & $30$ & $\frac{4}{29}$ & $190$ & $\vy_8$ & $134$\\
$E_8$ &  $p=2$ & $0$ & $30$ & $\frac{4}{29}$ & $190$ & $\vy_3$ & $136$ \\
\hline
\end{tabular}
\end{center}
%\end{table}
%\end{defn}


\begin{thebibliography}{GLS 99}

\bibitem[ANS03]{kn::ANS} M. Ab\'ert, N. Nikolov, B. Szegedy, 
\emph{Congruence subgroup growth of arithmetic groups in positive
  characteristic}, Duke Math. J., 
{\bf 117} (2003), no. 2, 367--383.
  
 \bibitem[BG01]{kn::BG} Y. Barnea, R. Guralnick, 
\emph{Subgroup growth in some pro-$p$ groups},
Proc. Amer. Math. Soc.
{\bf 130} (2001), no. 3, 653--659.
 
\bibitem[Bo68]{kn::Bo} N. Bourbaki, 
\emph{Groupes et alg\'{e}bres de Lie, Ch. IV--VI},
Hermann, Paris, 1968.

%\bibitem[Ca]{kn::Ca} Carter, Lie Algebras of Finite and Affine Type

\bibitem[H95]{kn::Hum} J. Humphreys, 
\emph{Conjugacy Classes in Semisimple Algebraic Groups}, 
Math. Surveys and Monographs, v.43, 
Amer. Math. Soc., Providence, 1995.

\bibitem[KW]{kn::KW} V. Kac, B. Weisfeiler, 
\emph{Coadjoint action of a semi-simple algebraic group and the center
  of the enveloping algebra in characteristic $p$},
Indag. Math. ,
{\bf 38}  (1976), no. 2, 136--151. 

\bibitem[LM91]{kn::LM91} A. Lubotzky, A. Mann, 
\emph{On groups of polynomial subgroup growth}, 
Invent. Math, 
{\bf 104} (1991), no. 3, 521--533 .
 
\bibitem[LSh94]{kn::LSh}
A. Lubotzky, A. Shalev,
\emph{On some {$\Lambda$}-analytic pro-{$p$} groups}, 
Israel J. Math.,
{\bf 85} (1994), 
no.~1--3, 307--337.

\bibitem[LSe03]{kn::LS}
A. Lubotzky, D. Segal, \emph{Subgroup growth}, Progress in
  Mathematics, vol. 212, Birkh\"auser, Basel, 2003.

\bibitem[P95]{Pr}
A. Premet, 
\emph{An analogue of the Jacobson-Morozov theorem for Lie algebras of
  reductive groups of good characteristic},
Trans. Amer. Math. Soc.  {\bf 347}  (1995), no. 8, 2961--2988.

\bibitem[Sh92]{kn::Sh92}
A. Shalev, \emph{Growth functions, $p$-adic analytic groups, and
  groups of finite coclass}, J. London Math. Soc.,
{\bf 46} (1992), no. 1, 111--122.

\bibitem[V05]{kn::VIGRE}
University of Georgia VIGRE Algebra Group, \emph{Varieties of
  nilpotent elements for simple Lie algebras II: Bad primes}, 
J. Algebra, 
{\bf 292} (2005), no. 1, 65--99.

\bibitem[W99]{kn::Wa}
W. Wang, \emph{Dimension of a minimal nilpotent orbit}, 
Proc. Amer. Math. Soc., 
{\bf 127} (1999), no. 3,  935--936. 

\end{thebibliography}
\end{document}